\documentclass[11pt, a4paper]{article}
\usepackage{authblk} 

\usepackage[utf8]{inputenc}
\usepackage[english]{babel}
\usepackage[toc,page]{appendix}
\usepackage{amsmath,amssymb,mathrsfs,amsthm,dsfont,amscd}
\usepackage{mathtools}
\usepackage[lofdepth,lotdepth]{subfig}

\usepackage{bbm}  

\usepackage{tikz}
\usetikzlibrary{matrix}
\usetikzlibrary{cd}
      
\usepackage{enumitem, hyperref}  
\makeatletter
\def\namedlabel#1#2{\begingroup  
    #2%
    \def\@currentlabel{#2}%
    \phantomsection\label{#1}\endgroup
}
\makeatother

\newtheorem{theorem}{Theorem}[section]

\newtheorem{definition}[theorem]{Definition}
\newtheorem{corollary}[theorem]{Corollary}
\newtheorem{lemma}[theorem]{Lemma}
\newtheorem{remark}[theorem]{Remark}

\numberwithin{equation}{section}

\newcommand {\rd}{\mathrm{d}}
\newcommand{\eps}{\varepsilon}
\newcommand{\R}{{\mathbb{R}}}
\newcommand{\N}{{\mathbb{N}}}

\renewcommand{\P}{{\mathbb{P}}}
\newcommand{\E}{{\mathbb{E}}}

\newcommand{\loiinit}{\nu_0}

\renewcommand{\O}{\mathcal{O}}
\renewcommand{\L}{\mathcal{L}}
\newcommand{\Cov}{\mathrm{Cov}}

\allowdisplaybreaks[1]

\usepackage{todonotes}

\newcommand*{\transp}[2][-3mu]{\ensuremath{\mskip1mu\prescript{\smash{\mathrm t\mkern#1}}{}{\mathstrut#2}}}%

\begin{document}
\title{Asymptotic behavior of a network of  neurons with random linear interactions}

\author[$\dagger$, $\star$]{Olivier Faugeras}
\author[$\dagger$, $\star$]{\'{E}milie Soret}
\author[$\star$]{Etienne Tanr\'{e}}
\affil[$\dagger$]{ Universit\'e C\^ote d’Azur, INRIA, France (team Mathneuro).}
\affil[$\star$]{ Universit\'e C\^ote d’Azur, INRIA, France (team Tosca).}

\maketitle

\begin{abstract}
We study the asymptotic behavior for asymmetric neuronal dynamics in a network 
of linear Hopfield neurons. 
The interaction between the neurons is modeled  by random 
couplings which are centered \textit{i.i.d.} random variables with finite 
moments of all orders. We prove that if the initial condition of the network is 
a set of  \textit{i.i.d.} random variables 
and independent of the synaptic weights, each component of the limit system is 
described as the sum of the corresponding coordinate of the  initial condition 
with a centered Gaussian process whose covariance function can be described in 
terms of a modified Bessel function. This process is not Markovian. The 
convergence is in law almost surely with respect to the random weights. Our method is 
essentially based on  the method of moments to obtain a  Central Limit Theorem.
\end{abstract}

\noindent
{\em AMS Subject of Classification (2020)}:\\
 60F10, 60H10, 60K35, 82C44, 82C31, 82C22, 92B20
\section{Introduction}\label{Sec:Intro}
We revisit the problem of characterizing the limit of a network of 
Hopfield neurons. Hopfield \cite{hopfield:82} defined a large class of 
neuronal networks and characterized some of their computational
properties \cite{hopfield:84,hopfield-tank:86}, i.e. their ability to 
perform computations. Inspired by his work, Sompolinsky and 
co-workers studied the thermodynamic limit of these networks when 
the interaction term is linear \cite{crisanti-sompolinsky:87} using the 
dynamic mean-field theory developed in 
\cite{sompolinsky-zippelius:82} for symmetric spin glasses. The 
method they use is a functional integral formalism used in particle 
physics and produces the self-consistent mean-field equations of the 
network. This was later extended to the case of a nonlinear interaction 
term, the nonlinearity being an odd sigmoidal function 
\cite{sompolinsky-crisanti-etal:88}. Using the same formalism the 
authors established the self-consistent mean-field equations of 
the network and the dynamics of its solutions which featured a chaotic 
behavior for some values of the network parameters. A little later the 
problem was picked up again by mathematicians. Ben Arous and 
Guionnet applied large deviation techniques to study the
thermodynamic limit of a network of spins interacting linearly with  
\textit{i.i.d.} centered Gaussian weights. The intrinsic spin dynamics 
(without interactions) is a stochastic differential equation. They prove 
that the annealed (averaged) law of the empirical measure satisfies a 
large deviation principle and that the good rate function of this large 
deviation principle achieves its minimum value at a unique non 
Markovian measure~\cite{guionnet:95,ben-arous-guionnet:95,guionnet:97}. 
They also prove averaged propagation of chaos results. Moynot and 
Samuelides \cite{moynot-samuelides:02} adapt their work to the case 
of a network of Hopfield neurons with a nonlinear interaction term, the 
nonlinearity being a sigmoidal function, and prove similar results in the 
case of discrete time. The intrinsic neural dynamics is the gradient of a 
quadratic potential. Our work is in-between that of Ben Arous and 
Guionnet and Moynot and Samuelides: we consider a network of 
Hopfield neurons, hence the intrinsic dynamics is simpler than the 
one in Ben Arous and Guionnet's case, with linear interaction between 
the neurons, hence simpler than the one in Moynot and Samuelides' 
work. We do not make the hypothesis that the interaction (synaptic) 
weights are Gaussian unlike the previous authors. The equations of our 
network are linear and therefore their solutions can be expressed 
analytically. As a consequence of this, we are able to use variants of the CLT and 
the 
moments method to characterize in a simple way the thermodynamic 
limit of the network without the tools of the theory of large deviations. 
Our main result is that the solution to the network equations converges 
in law toward a non Markovian process, sum of the initial condition and 
a centered Gaussian process whose covariance is characterized by a 
modified Bessel function. 

\subsection*{Plan of the paper}
We introduce the precise model in Section~\ref{sec:model}. In 
Section~\ref{Sec:MainResult- without noise case}, we state and prove 
our main result (Theorem~\ref{thm:cv}) on the asymptotic behavior of 
the dynamics in the absence of additive white noise. 
Section~\ref{Sec:Noise} is devoted to the general case with additive 
white noise 
	(Theorem~\ref{thm:cvnoise}). Our approach in establishing these results is 
	``syntactic'', based on Lemmas~\ref{lem:1} and \ref{lem:nisodd}.

\section{Network model}\label{sec:model}
We consider a network of $N$ neurons in interaction. Each neuron $i
\in \{1,\cdots,N\}$ is characterized by its membrane potential 
$(V^{i,(N)}(t))_t$  where $t\in \R_+$ represents  the time. The 
membrane potentials  evolve according to the system of stochastic 
differential equations
\begin{equation}\label{eq:EDS}
\left\lbrace
\begin{aligned}
V^{i,(N)}(t)&=V_0^i-\lambda\int_0^t V^{i,(N)}(s) \rd s +\frac{1}{\sqrt{N}}\sum_{j=1}^N \int_0^tJ_{i,j}^{(N)}V^{j,(N)}(s)\rd s + \gamma  B^{i}(t), \quad \forall i\in \{1,\cdots,N\}\\
\L(V_0^{(N)})&=\loiinit^{\otimes N},
\end{aligned}
\right.
\end{equation}
where $V_0^{(N)}=\left(V^{1}_0,\cdots,V^{N}_0\right)$ is the vector of 
initial conditions. The matrix $J^{(N)}$ is a square 
matrix of size $N$ and contains the synaptic weights. For $i\neq j$, 
the coefficient $J^{(N)}_{i,j}/\sqrt{N}$ 
represents the synaptic 
weight for pre-
synaptic neuron $j$ to post-synaptic neuron $i$. 
	The coefficient $J^{(N)}_{i,i}/\sqrt{N}$ can be seen as describing the interaction 
	of the neuron $i$ with itself. It turns out that it has no role in defining the mean 
	field limit. The parameters
$\lambda$ and $\gamma$ are constants. The $(B^{i}(t))_t$, $i\in 
\{1,\cdots,N\}$ are $N$ independent standard Brownian motions 
modelling the internal noise of each neuron. The initial condition is a 
random vector with \textit{i.i.d.} coordinates, each of distribution 
$\loiinit$. 

We denote by  \(V^{(N)}(t)\) the vector \((V^{1,(N)}(t),\cdots,V^{N,(N)}(t))\). Hence, we can write the  system~\eqref{eq:EDS} 
in  matrix form:
\begin{equation}\label{eq:eds_matrix}
\left\lbrace
\begin{aligned}
V^{(N)}(t) &=V_0^{(N)}-\int_0^t\lambda V^{(N)}(s) \rd s 
+\int_0^t\frac{J^{(N)}}{\sqrt{N}}V^{(N)}(s)\rd s+ \gamma B(t)\\
\L\left(V_0^{(N)}\right)&=\loiinit^{\otimes N}.
\end{aligned}
\right.
\end{equation}
System~\eqref{eq:eds_matrix} can be solved explicitly and its solution $V^{(N)}(t)$ is given by 
\begin{equation}\label{eq:Vtsoluceexplicit}
V^{(N)}(t)=e^{-\lambda 
t}\left[\exp\left(\frac{J^{(N)}}{\sqrt{N}}t\right)V_0^{(N)}+\gamma\int_0^t e^{\lambda 
s}\exp\left(\frac{J^{(N)}}{\sqrt{N}}(t-s)\right)\rd B(s)\right], \quad \forall t\in \R_+.
\end{equation}
For the rest of the paper, we make the following hypotheses on the distributions of 
$V_0^{(N)}$ and $J^{(N)}$.
\begin{description}
\item[\namedlabel{itm:H1}{(H1)} ]  $\loiinit$  is of compact support and we note
\[
\mu_0 := \int_\R x \, d\loiinit(x) \quad \text{and} \quad \phi_0 = \int_\R x^2 \, d\loiinit(x),
\]
 	its first and second order moments.
\item[\namedlabel{itm:H2}{(H2)} ]The elements of the  
matrix $J^{(N)}$ are \textit{i.i.d.} 
centered  and bounded random variables  of variance  $\sigma^2$. They are 
independent of the initial condition.
\end{description}

\section{Convergence of the particle system without additive Brownian noise (\(\gamma = 0\))}\label{Sec:MainResult- without noise case}
In this section, we consider the model without any additive noise, that is  
$\gamma=0$ in \eqref{eq:EDS}: The unique source of 
randomness in the dynamics comes from the random matrices 
$(J^{(N)})$ describing the synaptic weights, and the initial condition. 
\subsection{Mean field limit}\label{Subsec:meanfield}
 The following result describes the convergence  when $N\to +\infty$  of the coordinates of the vector 
$(V^{(N)}(t))_{t\in \R_+}$ to a Gaussian process whose 
covariance is determined by a Bessel function. Theorem~\ref{thm:cv} below can be seen as a kind of mean-field description 
of~\eqref{eq:EDS} as the number of neurons tends to infinity. 

\begin{theorem}\label{thm:cv}
Under the hypotheses \ref{itm:H1} and \ref{itm:H2}, for each $k\in \N^*$, the process 
$(V^{k,(N)}(t))_{t\in \R_+}$ converges in law to $(V^{k,(\infty)}(t))_{t\in \R_+}$ where, 
\[
V^{k,(\infty)}(t)=e^{-\lambda t}\left[V_0^{k}+Z^{k}(t)\right], \qquad \forall t\in \R_+.
\]
The process \(\left(Z^k(t)\right)_{t\in \R_+}\) is a centered Gaussian process starting from $0$ ( $Z^k(0)=0$) such that
\begin{equation}\label{eq:I0tilde}
\mathbb{E}\left[Z^k(t)Z^k(s)\right] = \phi_0 \widetilde{I}_0(2\sigma \sqrt{st}), \qquad \text{where }\quad 
\widetilde{I}_0(z)=\sum_{\ell\geq 1} z^{2\ell}/(2^{2\ell}(\ell !)^2).
\end{equation}

Moreover, for all $t\in \R_+$, $Z^{k}(t)$ is independent of $V_0^k$.
\end{theorem}

\begin{remark}
 The function $\widetilde{I}_0$ is closely connected to
  the modified Bessel function of the first kind \(I_0\), defined as a solution of the ordinary differential equation $z^2y''+zy'-z^2y=0$, $y'=dy/dz$. This function is the 
sum of the series $(z^{2\ell}/(2^{2\ell}(\ell !)^2))_{\ell \geq 0}$ which is absolutely convergent for all $z \in \mathbb{C}$, i.e.: 
$I_0(z)=\sum_{\ell\geq 0} (z/2)^{2\ell}/(\ell !)^2$, 
so that we have
\[\widetilde{I}_0(z)=I_0(z)-1.\]
\end{remark}
The proof of Theorem~\ref{thm:cv} requires the following lemma.
\begin{lemma}\label{lem:1}
Let $J$ be an infinite matrix such that its finite restrictions \(J^{(N)}\) satisfy
\ref{itm:H2} and consider a sequence of 
 bounded
(not necessarily identically distributed)
random variables $\left(Y_{j}\right)_{j\in \N^*}$, independent of \(J\).

Assume that, almost surely,
\begin{equation}\label{eq:assumptionV0}
\lim_{N \to \infty}  \frac{1}{N} \sum_{j=1}^N Y_j^2 :=\phi<+\infty.
\end{equation}
 For all $\ell\in \N^*$ and for all $1\leq k\leq N$, define
 \[U_\ell^{k,(N)}:=\transp{e_k}\left((J^{(N)})^\ell Y^{(N)}\right),\]
 where $e_k$ is the $k$-th vector of the standard basis of $\R^N$.
Then, for all $1 \leq \ell_1<\cdots<\ell_m$, $m \in \N^*$,  the vector 
\[
\left(\frac{1}{\sqrt{N}^{\ell_1}}U_{\ell_1}^{k,(N)}, \cdots , 
\frac{1}{\sqrt{N}^{\ell_m}}U_{\ell_m}^{k,(N)}\right)
\]
converges in law as $N\to +\infty$, to an $m$-dimensional Gaussian random vector of diagonal 
covariance matrix $\mathrm{diag}(\sigma^{2\ell_i}\phi)$ independent of any finite 
subset of the sequence $\left(Y_{j}\right)_{j\in \N^*}$.
\end{lemma}
\begin{remark}\label{rem:ind}
The hypothesis that the $Y_j$s are bounded is not the only possible one. The Lemma is also true for independent $Y_j$s with finite moments of all orders.
\end{remark}
  We give one corollary of this Lemma.
\begin{corollary}\label{coro:pcomponents}
Under the same assumptions as in Lemma~\ref{lem:1},
 for all integers $p > 1$ and  $1 \leq k_1 < k_2 < \cdots < k_p$ the vector
 \begin{multline*}
\bigg(\frac{1}{\sqrt{N}^{\ell_1}}U_{\ell_1}^{k_1,(N)}, \cdots , 
\frac{1}{\sqrt{N}^{\ell_m}}U_{\ell_m}^{k_1,(N)},\\ \frac{1}{\sqrt{N}^{\ell_1}}U_{\ell_1}^{k_2,(N)}, \cdots , 
\frac{1}{\sqrt{N}^{\ell_m}}U_{\ell_m}^{k_2,(N)}, \cdots, \frac{1}{\sqrt{N}^{\ell_1}}U_{\ell_1}^{k_p,(N)}, \cdots , 
\frac{1}{\sqrt{N}^{\ell_m}}U_{\ell_m}^{k_p,(N)}\bigg)
\end{multline*}
converges in law as $N\to +\infty$, to an $mp$-dimensional Gaussian random vector of diagonal 
covariance matrix $\mathrm{diag}(\sigma^{2\ell_i}\phi)$, $i=1,\cdots,m$ repeated $p$ 
times, independent of any finite subset of the sequence $\left(Y_{j}\right)_{j\in \N^*}$.
\end{corollary}
\begin{proof}
It is easy to adapt the proof of Lemma~\ref{lem:1}.
\end{proof}
\begin{proof}[Proof of Lemma~\ref{lem:1}]\ \\
W.l.o.g, we consider the case \(k=1\) and do not show the index \(1\) in the 
	proof, i.e. we write \(U_\ell^{(N)}\) for  \(U_\ell^{1,(N)}\).
  We first prove by the method of moments that $\frac{1}{N^{\ell/2}} U_\ell^{(N)}$ converges in law when $N \to \infty$ toward a centered Gaussian random variable of variance $\sigma^{2\ell}\phi$, i.e. the case $m=1$ of the Lemma. We then sketch the generalization of the proof to the  case $m > 1$.
  
To do this we expand \(U_\ell^{(N)}\) and write
\[
\frac{1}{N^{n\ell/2}} \left( U_\ell^{(N)} \right)^n= \frac{1}{N^{n\ell/2}} 
\sum_{j^1,\cdots,j^n} J_{1,j^1_1}^{(N)}
J_{j^1_1,j^1_2}^{(N)} \cdots J_{j^1_{\ell-1},j^1_\ell}^{(N)} \cdots J_{1,j^n_1}^{(N)} 
J_{j^n_1,j^n_2}^{(N)} \cdots 
J_{j^n_{\ell-1},j^n_\ell}^{(N)}\,Y_{j^1_\ell}\cdots Y_{j^n_\ell},
\]
where $\sum_{j^1,\cdots,j^n}$ means $\sum_{\underset{\underset{j^n_1,\cdots,j^n_\ell 
=1}{\cdot}}{j_1^1,\cdots,j^1_\ell}=1}^N$.

We follow and recall the notations of \cite{agz}: we denote $j^r$ the 
sequence (word) of $\ell+1$ indexes $(1,j^r_1,\cdots,j^r_\ell)$.
To each word $j^r$   we associate its length 
$\ell+1$, its support $supp(j^r)$, the set of different integers in $\{1,\cdots,N\}$ in 
$j^r$, and its weight $wt(j^r)$, the cardinality of  $supp(j^r)$. We also associate the 
graph $G_{j^r}=(V_{j^r},E_{j^r})$ where $V_{j^r}=supp(j^r)$ has  by definition 
$wt(j^r)$ vertices and the edges are constructed by walking through $j^r$ from left to 
right. Edges are oriented if they connect two different indexes. In detail we have
\[
E_{j^r}=\{(1,j_1^r) \cup(j_i^r,j_{i+1}^r), i=1,\cdots \ell-1\}
\]

Each edge $e \in 
E_{j^r}$ has a weight noted $N_e^{j^r}$ which is the number of times it is traversed 
by the sequence $j^r$.

We note $\mathbf{j}$ 
the sentence $(j^1,\cdots,j^n)$ of $n$ words $j^r$, $r=1,\cdots,n$, of 
length $\ell+1$ and we associate to it the 
graph $G_{\mathbf{j}}=(V_{\mathbf{j}},E_{\mathbf{j}})$ obtained by piecing together 
the $n$ graphs $G_{j^r}$. In detail, as with words, we define for a sentence 
$\mathbf{j}$ its support $supp(\mathbf{j})=\cup_{r=1}^n supp(j^r)$ and its weight 
$wt(\mathbf{j})$ as the cardinality of $supp(\mathbf{j})$. We then 
set 
$V_{\mathbf{j}}=supp(\mathbf{j})$ and $E_{\mathbf{j}}$ the set of edges. Edges are 
directed if they connect two different vertices and we have
\[
E_{\mathbf{j}} = \{(1,j^r_1) \cup (j^r_i,j^r_{i+1}  ),\, i=1,\cdots \ell-1,\, r=1,\cdots,n  \}
\]

Two sentences $\mathbf{j}_1$ and $\mathbf{j}_2$ are equivalent, noted $\mathbf{j}_1 
\simeq \mathbf{j}_2$ if there exists a bijection on $\{1,\cdots,N\}$ that maps one into 
the other.

For $e \in E_{\mathbf{j}}$ we note $N_e^{\mathbf{j}}$ the number of times $e$ is traversed by the union of the sequences $j^r$.

By independence of the elements of the matrix $J$, the independence of the $J$s 
and the $Y$s, and by construction of the graph $G_{\mathbf{j}}$ we have
\begin{equation}\label{eq:eqsent}
\E\left[  
\frac{1}{N^{n\ell/2}} \left( U_\ell^{(N)} \right)^n  
\right] =   
\sum_{\mathbf{j}} \frac{1}{N^{n\ell/2}} 
\prod_{e \in E_{\mathbf{j}}} 
\E \left[ 
\left(J_{1,1}^{(N)}\right)^{N_e^{\mathbf{j}}}  \right]  \E[Y_{j^1_\ell}\cdots Y_{j^n_\ell}]
:= \sum_{\mathbf{j}} 
T_{\mathbf{j}}
\end{equation}
In order for $T_{\mathbf{j}}$ to be non zero we need to enforce $N_e^{\mathbf{j}} 
\geq 2$ for all $e \in E_{\mathbf{j}}$. This implies
\[
n \ell = \sum_{e \in E_{\mathbf{j}}} N_e^{\mathbf{j}} \geq 2 \left| E_{\mathbf{j}}  
\right| \geq 2(wt(\mathbf{j})- 1) ,
\]
i.e.
\[
wt(\mathbf{j}) \leq \lfloor n\ell/2 \rfloor+1.
\]
	We now make the following definitions: 
	\begin{definition}\label{def:Wlnt}
		Let $\mathcal{W}_{\ell,n,t}$ be the set of representatives of equivalent classes 
		of 
		sentences $\mathbf{j}$, $\mathbf{j}=(j^1,\cdots,j^n)$ of $n$ words of length 
		$\ell+1$ starting with 1, such that $t=wt(\mathbf{j})$ and  
		$N_e^{\mathbf{j}} \geq 2$ for all $e \in E_{\mathbf{j}}$.
	\end{definition}
	\begin{definition}\label{def:Alnt}
		Let $\mathcal{A}_{\ell,n,t}$ be the set of 
		sentences $\mathbf{j}$, $\mathbf{j}=(j^1,\cdots,j^n)$ of $n$ words of length 
		$\ell+1$ starting with 1, such that $t=wt(\mathbf{j})$ and  
		$N_e^{\mathbf{j}} \geq 2$ for all $e \in E_{\mathbf{j}}$.
              \end{definition}
We now rewrite \eqref{eq:eqsent}
\begin{multline}\label{eq:eqsent1}
\E\left[  
\frac{1}{N^{n\ell/2}} \left( U_\ell^{(N)} \right)^n  
\right] =  
\frac{1}{N^{n\ell/2}} \sum_{t=1}^{\lfloor 
  n\ell/2 \rfloor+1} 
   \sum_{\mathbf{j} \in \mathcal{A}_{\ell,n,t}} \prod_{e \in 
    E_{\mathbf{j}}} \E \left[ \left(J_{1,1}^{(N)}\right)^{N_e^{\mathbf{j}}}  \right]
  \E[Y_{j^1_\ell}\cdots Y_{j^n_\ell}]  \\
= \sum_{t=1}^{\lfloor n\ell/2 \rfloor+1}\sum_{\mathbf{j}' \in \mathcal{W}_{\ell,n,t}} 
\sum_{\mathbf{j} \simeq 
\mathbf{j}'} 
\frac{1}{N^{n\ell/2}} 
\prod_{e \in E_{\mathbf{j}}} 
\E \left[ 
\left(J_{1,1}^{(N)}\right)^{N_e^{\mathbf{j}}}  \right]  \E[Y_{j^1_\ell}\cdots Y_{j^n_\ell}].
\end{multline}
Our assumption on $Y$ ensures that 
\[
\left|\E[Y_{j^1_\ell}\cdots Y_{j^n_\ell}]\right| \leq K 
\]
for a constant \(K\) independent of \(\mathbf{j}\) and $N$.
In addition, 
for \(\mathbf{j} \simeq \mathbf{j}' \), we have
\[
\prod_{e \in E_{\mathbf{j}}} 
\E \left[ 
\left(J_{1,1}\right)^{N_e^{\mathbf{j}}}  \right] 
=
\prod_{e \in E_{\mathbf{j}'}} 
\E \left[ 
\left(J_{1,1}\right)^{N_e^{\mathbf{j}'}}  \right] .
\]
We deduce that for any \(\mathbf{j}' \in \mathcal{W}_{\ell,n,t}\) hence such that 
$wt(\mathbf{j}')=t$,
\[
\left|\sum_{\mathbf{j} \simeq \mathbf{j}'} 
\frac{1}{N^{n\ell/2}} 
\prod_{e \in E_{\mathbf{j}}} 
\E \left[ 
\left(J_{1,1}\right)^{N_e^{\mathbf{j}}}  \right]  \E[Y_{j^1_\ell}\cdots Y_{j^n_\ell}]
\right|
\leq K \frac{C_{N,t}}{N^{n\ell/2}}
\left|
\prod_{e \in E_{\mathbf{j}'}} 
\E \left[ 
\left(J_{1,1}\right)^{N_e^{\mathbf{j}'}}  \right]\right|,
\]
where
\begin{equation}\label{eq:Cnt}
  C_{N,t}=(N-1)(N-2)\cdots(N-t+1)  \simeq  
  N^{t-1}
\end{equation}
is the number of sentences that are equivalent to a given sentence $\mathbf{j}'$, 
with 
weight
$t$ (remember that the 
first element of each word is equal to 1).
Since the cardinality of $\mathcal{W}_{\ell,n,t}$ is independent of $N$, we conclude 
that each term in the right hand side of \eqref{eq:eqsent1} is upper bounded by a 
constant independent of $N$ times the ratio $C_{N,t}/N^{n\ell/2}$. According to 
\eqref{eq:Cnt}, this is equivalent to  $N^{t - n\ell/2-1}$. Therefore, asymptotically in $N$, the 
only relevant term in the right hand side of \eqref{eq:eqsent1} is the one 
corresponding to $t=\lfloor n\ell/2 \rfloor+1$. 

In order to proceed we use the following Lemma whose proof is postponed.
\begin{lemma}\label{lem:nisodd}
If	$\mathcal{W}_{\ell,n,t}
\neq 
\emptyset$ and \(n\) is odd, then \(t \leq \lfloor n\ell/2 \rfloor\).
\end{lemma}

      \noindent
We are ready to complete the proof of Lemma \ref{lem:1}.\\
If $n$ is odd, Lemma \ref{lem:nisodd} shows that $t \leq \lfloor n\ell/2 \rfloor$, so 
that the maximum value of $C_{N,t}$ in \eqref{eq:eqsent1} is $\O( N^{\lfloor n\ell/2 
\rfloor -1})$ and we have $\lim_{N \to \infty}\frac{1}{N^{n\ell/2}} 
\E[(U_\ell^{(N)})^n]=0$. 
If $n=2p$ is even, \eqref{eq:eqsent1} commands
\begin{equation}\label{eq:eqsent2}
  \E\left[ \frac{1}{N^{n\ell/2}} \left( U_\ell^{(N)} \right)^n  \right] \simeq  \frac{1}{N^{p\ell}} 
   \sum_{\mathbf{j} \in \mathcal{A}_{\ell,2p, 
  p\ell +1}} \prod_{e \in 
    E_{\mathbf{j}}} \E \left[ \left(J_{1,1}^{(N)}\right)^{N_e^{\mathbf{j}}}  \right]
  \E[Y_{j^1_\ell}\cdots Y_{j^{2p}_\ell}]
\end{equation}
An element of $\mathcal{A}_{\ell,2p,p\ell+1}$ is a set of $p$ pairs of identical words 
of length $\ell+1$. Each word in a pair contains $\ell+1$ different symbols and the 
intersection of their $p$ supports is equal to $\{1\}$. Thus we have $\E \left[ \left(J_{1,1}\right)^{N_e^{\mathbf{j}}}  \right]=\sigma^{2p\ell}$ for all $\mathbf{j} \in \mathcal{A}_{\ell,2p, p\ell +1}$.

We have
\[
 \E\left[ \frac{1}{N^{n\ell/2}} \left( U_\ell^{(N)} \right)^n  \right] \simeq  \frac{\sigma^{2p\ell}}{N^{p\ell}} 
   \sum_{\mathbf{j} \in \mathcal{A}_{\ell,2p, 
       p\ell +1}}
  \E[Y_{j^1_\ell}\cdots Y_{j^{2p}_\ell}]
\]
There are $(2p-1)!! (= 1.3.\cdots (2p-1))$ 
ways to group the $2p$ words in pairs. Indeed, given a sequence of $p\ell+1$ 
different symbols, we set $q_1=1$ and pick the first $\ell$ symbols in the sequence  
to obtain $j^{q_1}$. We then choose one, say $j^{r_1}$, $r_1 \neq q_1$, among the 
$n-1=2p-1$ remaining words to form the first pair of identical words. We next go to 
the first unpaired word after $j^{q_1}$, say $j^{q_2}$, $q_2 \notin \{q_1,r_1\}$ and 
choose one among the $n-3=2p-3$ remaining words to form the second pair of 
identical words (with the next $\ell$ symbols in the sequence). And this goes on 
until we reach the end, having exhausted the sequence of $p\ell+1$ different 
symbols. It follows that there are $(2p-1)!!$ ways to group the $2p$ indexes in pairs.

For each such grouping and for each $p$-tuple of different indexes $(j^1_\ell,\cdots,j^p_\ell)$ there are $(N-(p+1))(N-(2p+1))\cdots (N-((\ell-1) p+1))$ ways of choosing the remaining $\ell-1$ $p$-tuples of different indexes $(j^1_k,\cdots,j^p_k)$, $k=1,\cdots,\ell-1$. Putting all this together we obtain
\begin{multline*}
  \E\left[ \frac{1}{N^{n\ell/2}} \left( U_\ell^{(N)} \right)^n  \right] \simeq (2p-1)!!\,\sigma^{2p\ell} \frac{1}{N^p}
  \sum_{j^1,\cdots,j^p=2,\text{ all indexes different}}^N \E [Y_{j^1}^2\cdots Y_{j^{p}}^2] 
  \\
 \simeq (2p-1)!! \sigma^{2p\ell} \frac{1}{N^p}
  \sum_{j^1,\cdots,j^p=1}^N \E [Y_{j^1}^2\cdots Y_{j^{p}}^2]  \\
\simeq (2p-1)!!\, \sigma^{2p\ell}  \mathbb{E}\left[
\left(\dfrac{1}{N}\sum_{j=1}^N Y_{j}^2\right)^p\right].
\end{multline*}
By \eqref{eq:assumptionV0} and dominated convergence we obtain
\[
  \lim_{N \to \infty} \E\left[ \frac{1}{N^{n\ell/2}} \left( U_\ell^{(N)} \right)^n  \right] = 
  \begin{cases}
   0 \quad \text{ if } n \text{ odd}\\
   \sigma^{2\ell p}(2p-1)!! \phi^p \quad  \text{ if } n=2p.
  \end{cases}
\]
At this point we have proved that $\frac{1}{N^{\ell/2}} U_\ell^{(N)} $ converges in law 
when $N \to \infty$ to a centered Gaussian of variance $\sigma^{2l}\phi $.

We go on to sketch the proof  that for all $m \in \N^*$ and integers $1 \leq \ell_1 < 
\ell_2 < \cdots < \ell_m$, the $m$-dimensional vector 
$(\frac{1}{N^{\ell_1/2}}U_{\ell_1}^{(N)},\,\frac{1}{N^{\ell_2/2}}U_{\ell_2}^{(N)},\,\cdots,\, 
\frac{1}{N^{\ell_m/2}}U_{\ell_m}^{(N)})$ converges in law toward an $m$-dimensional 
centered Gaussian vector with covariance matrix ${\rm diag}(\sigma^{2\ell_k}\phi)$.

The proof is essentially the same as in the case $m=1$ but the notations become much heavier. The major ingredients are
\begin{enumerate}
\item To each $U_{\ell_k}^{(N),\,}$, $1 \leq k \leq m$ we associate the graphs $G_{\mathbf{j^k}}$, with $\mathbf{j^k}=(j^{k,1},\cdots,j^{k,n_k})$ and $j^{k,r}=\{1,j^{k,r}_1,\cdots,j^{k,r}_{\ell_k}\}$, $r=1,\cdots,n_k$.
\item We piece these graphs together to obtain the graph $G_{\mathbf{j}^1,\cdots,\mathbf{j}^m}$. This allows us to write a formula similar to \eqref{eq:eqsent} for $\E\left[\frac{1}{N^{(n_1\ell_1+\cdots+n_m\ell_m)/2}}\left(U_{\ell_1}^{(N)}\right)^{n_1}\cdots\left(U_{\ell_m}^{(N)}\right)^{n_m}\right]$.
\item We enforce the condition that edges in $G_{\mathbf{j}^1,\cdots,\mathbf{j}^m}$ must have a weight larger than or equal to 2.
\item We generalize the definition \ref{def:Wlnt} to the set noted $\mathcal{W}_{\ell_1,n_1,\cdots,\ell_m,n_m,t}$ and prove an analog to Lemma \ref{lem:nisodd}.
\item These last two steps allow us to write a formula analog to 
  \eqref{eq:eqsent1}.
\item Lemma \ref{lem:nisodd} can be easily generalized to the following
  \begin{lemma}\label{lem:nisodd1}
If	$\mathcal{W}_{\ell_1,n_1,\cdots,\ell_m,n_m,t} \neq \emptyset$ and if any of \(n_1,\cdots,n_m\) is odd, then \(t \leq \lfloor (n_1\ell_1+\cdots+n_m\ell_m)/2 \rfloor\).
\end{lemma}
\item Combining all this yields the result
  \begin{multline*}
    \lim_{N \to \infty}\E\left[\frac{1}{N^{(n_1\ell_1+\cdots+n_m\ell_m)/2}}\left(U_{\ell_1}^{(N)}\right)^{n_1}\cdots\left(U_{\ell_m}^{(N)}\right)^{n_m}\right] = \\
    \left\{
      \begin{array}{cc}
        0 & \text{if any of } n_1,\cdots,n_m \text{ is odd}\\
        \prod_{k=1}^m \sigma^{2\ell_k p_k}(2p_k-1)!! \phi^{p_k} & \text{if } 
        n_1=2p_1,\cdots,n_m=2p_m
      \end{array}
      \right.
  \end{multline*}
which shows that the $m$-dimensional vector  $(\frac{1}{N^{\ell_1/2}}U_{\ell_1}^{(N)},\,\frac{1}{N^{\ell_2/2}}U_{\ell_2}^{(N),\,}\cdots,\, \frac{1}{N^{\ell_m/2}}U_{\ell_m}^{(N)})$ converges in law toward an $m$-dimensional centered Gaussian vector with covariance matrix ${\rm diag}(\sigma^{2\ell_k}\phi)$.
\end{enumerate}
This ends the proof of Lemma~\ref{lem:1}. It is central to our approach and allows us 
to establish Theorems~\ref{thm:cv} and \ref{thm:cvnoise} in a  ``syntactic'' manner 
by connecting the stochastic properties of the matrices $J^{(N)}$ and the structure 
of the sequences of indexes that appear when raising them to integer powers.

Note that the proof also shows that this Gaussian vector is independent of any finite subset of the 
sequence $\left(Y_{j}\right)_{j\in \N^*}$.
Indeed, given $k$ distinct integers $j_1,\cdots,j_k$, if we eliminate from the previous construction all words ending with any of these integers, the limits will not change and will be, by construction, independent of $Y_{j_1},\cdots,Y_{j_k}$.
\end{proof}
\begin{proof}[Proof of Lemma~\ref{lem:nisodd}]
	For any \(r\in1,\cdots,n\), consider the number of times the letter \(j_1^r\) appears 
	at this position in the set of \(n\) words
	\[
	c_r:= \sum_{s = 1}^n \mathbbm{1}_{\{j^s_1 = j^r_1\}}.
	\]
	Since the number \(n\) of words is  odd, at least one of the \(c_r\) is also odd.
	Indeed, assume $c_1=2k_1$ is even. Consider the $n-2k_1$ indexes which 
	are not equal to $j_1^1$ and assume w.l.o.g. that one of them is 
	$j_1^{2k_1+1}$. If $c_{2k_1+1}$ is odd, we are done, else let 
	$c_{2k_1+1}=2k_2$ and consider the $n-2k_1-2k_2$ indexes which are not 
	equal to $j_1^{2k_1+1}$ and to $j_1^1$. This process terminates in a finite 
	number of steps and since $n$ is odd we must necessarily either find an 
	odd $c_r$ or end up with a single remaining element corresponding also to 
	an odd $c_r$.

	Assume 
	( w.l.o.g,)  \(c_1\) is odd.
	\begin{enumerate}
		\item[Case 1] If \(c_1 = 1\), there must be another oriented edge \((1,j_1^1)\) in 
		the sentence and it cannot be the first edge of any of the remaining $n-1$ 
		words. It means that at least one of the \(j^s_k\) is equal to \(1\). Hence
		there is an oriented edge  \((j,1)\), and it has to appear at least twice so that  
		there are
		at 	least two \(j^s_k\) equal to \(1\). The total number of letters is $n(\ell+1)$ 
		and we have shown that the letter 1 appeared at least $n+2$ times.
		
		Now, since
		
		\[
		supp(\mathbf{j}) = \{1\} \cup (supp(\mathbf{j}) \backslash\{1\}),
		\]
		and every letter in 
		\(supp(\mathbf{j}) \backslash\{1\}\) 
		has to appear at least twice, we conclude that
		\[
		2 |supp(\mathbf{j}) \backslash\{1\}| \leq n\ell -2,
		\]
		and therefore  \(wt(\mathbf{j}) - 1\leq (n\ell - 2)/2 \).
		
		\item[Case 2] If \(c_1 = 3\), w.l.o.g, we assume that \(j^1_1 = j^2_1 = j^3_1\) 
		\begin{enumerate}
			\item[\textbf{2a}] If \(j^1_2 = j^2_2 = j^3_2\), then we have
			
			\[
			2|supp (\mathbf{j})\backslash\{1,j^1_1,j^1_2\}| 
			\leq n(\ell +1) - \underbrace{n}_{1} -\underbrace{3}_{j_1^1} - 
			\underbrace{3}_{j_2^1},
			\]
			and hence  \(wt(\mathbf{j}) - 3 \leq (n\ell - 6)/2\).
			\item[\textbf{2b}] Otherwise, w.l.o.g., $j_2^1\neq j_2^2$ and $j_2^1 \neq 
			j_2^3$. 
			The 
			edge $(j_1^1,j_2^1)$ must appear twice and hence the letter $j_1^1$ appears 
			at 
			least 4 times, i.e.
			\[
			2(wt(\mathbf{j})-2) \leq n\ell -4
			\]
			
		\end{enumerate}
		
		\item[Case 3] If \(c_1\geq 5\), then \(2(wt(\mathbf{j}) - 2)\leq n\ell-5 \). 
                \end{enumerate}
                	This ends the proof of Lemma~\ref{lem:nisodd}. It is a good example of 
                	the use of our ``syntactic'' approach: it provides an  upper bound on 
                	the number of different symbols in a sentence of $n$ words of indexes 
                	when $n$ is odd which is key in establishing the convergence 
                	properties of the powers of the matrices $J^{(N)}$ when $N \to \infty$.
 \end{proof}
\begin{proof}[Proof of Theorem~\ref{thm:cv}]
Without loss of generality we assume $k=1$.  We first expand the exponential of the 
matrix $J^{(N)}$ 
in~\eqref{eq:Vtsoluceexplicit} (with $\gamma=0$) and express the first coordinate of $V^{(N)}(t)$ as 
\begin{equation}\label{eq:Vt1expand}
V^{1,(N)}(t)=
e^{-\lambda t}\left(V_0^{1}+\sum_{\ell\geq 1}\frac{t^\ell}{\ell!}\frac{1}{\sqrt{N}^\ell}
\sum_{j_1,\cdots,j_\ell=1}^N J_{1,j_1}^{(N)}J_{j_1,j_2}^{(N)}
\cdots J_{j_{\ell-1},j_\ell}^{(N)}V_0^{j_\ell}\right).
\end{equation}
The main idea of the proof is to truncate the infinite sum in the right hand side of \eqref{eq:Vt1expand} to order $n$, establish the limit of the truncated term when $N \to \infty$, obtain the limit of the result when $n \to \infty$, and show that it is the  limit of the non truncated term when $N \to \infty$.

Let $n\in \N^*$ and define the  partial sum \(Z^{(N)}_n(t)\) of order $n$ of \eqref{eq:Vt1expand}
\begin{equation}\label{eq:ZandU}
Z^{(N)}_n(t)=\sum_{\ell=1}^n\dfrac{t^\ell}{\ell ! }\, \dfrac{1}{\sqrt{N}^\ell}U_\ell^{(N)}
\qquad \mbox{ with } \quad U_\ell^{(N)}= \transp{e_1} \left(J^{(N)}\right)^\ell V_0.
\end{equation}
Lemma~\ref{lem:1} with $Y=V_0$ dictates the convergence in law, for all $n\in \N^*$, of the 
vector
\[\left(\dfrac{1}{\sqrt{N}}U_1^{(N)}, \cdots, \dfrac{1}{\sqrt{N}^n}U_n^{(N)}\right)_{N\in N^*}\]
to an $n$-dimensional centered Gaussian random vector, independent of $V_0^1$,  with a diagonal covariance matrix $\Sigma$, such that 
$\forall 1\leq i\leq n$, $\Sigma_{i,i}=\sigma^{2i}\phi_0$. It follows that we have 
the independence between the limits in law of each $\dfrac{1}{\sqrt{N}^\ell}U_\ell^{(N)}$, 
$1\leq \ell\leq n$ and also the convergence in law of the sum $Z_n^{(N)}(t)$ to a 
centered Gaussian random variable $Z_n(t)$, independent of $V_0^1$, of variance 
$\phi_0 \sum_{\ell =1 }^n \frac{(\sigma t)^{2\ell}}{(\ell !)^2}=:\widetilde{I}_{0,n}(2 
\sigma t)$.

Moreover, since the function $\widetilde{I}_{0,n}$ converges pointwise to 
$\widetilde{I}_0$ as \(n\) goes to infinity, the Kolmogorov–Khinchin Theorem 
(see e.g. \cite[Th. 1 p.6]{shiryaev:2019}) gives 
the convergence of \(Z_n(t)\).\\
 \[
 Z_n(t) \mathrel{\substack{\mathcal{L}\\\longrightarrow\\ n \to +\infty}} Z^1(t) \sim \mathcal{N}(0,\phi_0\widetilde{I}_{0}(2\sigma t))
 \]
 It is clear that $Z^{(N)}_n(t) \mathrel{\substack{\mathcal{L}\\\longrightarrow\\ n \to +\infty}} Z^{(N)}(t)$ where $Z^{(N)}(t)=\sum_{\ell\geq 1}\dfrac{t^\ell}{\ell!}U_\ell^{(N)}$,  so that we have 
\begin{figure}[h]
   \centering
   \begin{tikzpicture}[scale=2]
     \node (A) at (0,0) {$Z^{(N)}_n(t)$};
     \node (B) at (2,0) {$Z_{n}(t)$};
     \node (C) at (2,-1) {$Z^1(t)$};
     \node (D) at (0,-1) {$Z^{(N)}(t)$};
     \path[->]
     (A) edge node[above]{$\mathcal{L}$} node[below] {$N\to \infty $} (B)
     (A) edge node [left]{$\mathcal{L}$} node[right] {$n\to \infty $} (D)
     (B)  edge node[left]{$\mathcal{L}$} node[right] {$n\to \infty $} (C);
   \end{tikzpicture}
 \end{figure}\ \\
 It remains  to show that this diagram is commutative i.e. that $Z^{(N)}(t) \mathrel{\substack{\mathcal{L}\\\longrightarrow\\ N \to +\infty}} Z^1(t)$.
       
 According to 
 \cite[Th. 25.5]{billingsley1995}, to obtain the convergence in law of 
 $Z^{(N)}(t)$ to $Z^1(t)$ as $N\to +\infty$,  it is sufficient to show that for 
 all $t\in \R^*$
 \begin{equation}\label{eq:limnlimsupN}
 \lim_{n\to +\infty}\limsup_{N\to +\infty}\P\left[\mid (Z^{(N)}_n(t))-Z^{(N)}(t)\mid\geq \varepsilon\right]=0,\qquad \forall \eps>0.
 \end{equation}
 By Markov inequality, we have 
 \begin{equation}
 \P\left[\mid Z^{(N)}_n(t)-Z^{(N)}(t)\mid\geq \eps \right]\leq \frac{1}{\varepsilon}\E\left[\left| Z^{(N)}_n(t)-Z^{(N)}(t)\right|\right],
 \end{equation}
 and
 \begin{align*}
 \E\left[\mid Z^{(N)}_n(t)-Z^{(N)}(t)\mid\right] &=\E\left[\left|\sum_{\ell\geq n+1}\frac{t^\ell}{\ell!}U_\ell^{(N)}\right|\right]\\
 &\leq \sum_{\ell\geq n+1}\dfrac{t^\ell}{\ell!}\E\left[\left| U_\ell^{(N)}\right|\right].
 \end{align*}
 Moreover, by step 1, $\left(U_\ell^{(N)}\right)_N$ converges in law as $N\to \infty$ to a centered Gaussian random variable $U_\ell$ of variance $\phi_0 \sigma^{2\ell}$. Then  the law of $|U_\ell|$ is a half normal distribution, and hence 
 \begin{align*}
\E\left[|U_\ell|\right]=\dfrac{\sqrt{2\phi_0}}{\sqrt{\pi}}\sigma^{\ell}.
 \end{align*}
 Then, 
 \[
 \limsup_{N\to+\infty}\E\left[\mid Z^{(N)}_n(t)-Z^{(N)}(t)\mid\right]\leq \frac{\sqrt{2\phi_0}}{\sqrt{\pi}}\sum_{\ell\geq n+1}\dfrac{t^\ell}{\ell!}\sigma^\ell.
 \]
 The right hand side goes to zero as $n \to \infty$ and \eqref{eq:limnlimsupN} follows, hence we have obtained the convergence in law of $V^{1,(N)}(t)$ to $e^{-\lambda t}\left[V_0^1+Z^1(t)\right]$, $Z^1(t)$ being independent of $V_0^1$.
 
   In order to prove that the process $(Z^1(t))_{t\in \R_+}$ is Gaussian we show that $aZ^1(t)+bZ^1(s)$ is Gaussian for all reals $a$ and $b$ and all $s,\,t \in \R_+$. But this is clear from \eqref{eq:ZandU} which shows that
   \[
aZ^{(N)}_n(t)+bZ^{(N)}_n(s)=\sum_{\ell=1}^n\dfrac{at^\ell+bs^\ell}{\ell ! }\, \dfrac{1}{\sqrt{N}^\ell}U_\ell^{(N)},
\]
and the previous proof commands that $aZ^{(N)}(t)+bZ^{(N)}(s)$ converges in law when $N \to \infty$ toward a centered Gaussian random variable of variance $\phi_0(a^2\widetilde{I}_0(2\sigma t)+2ab \widetilde{I}_0(2\sigma \sqrt{ts})+b^2\widetilde{I}_0(2\sigma s))$.
\end{proof}
We now make a few remarks concerning the properties of the mean field limit.
\begin{remark}[Decomposition as an infinite sum of independent standard Gaussian variables]
  A consequence of the proof is that $Z^1(t)$ is equal in law to the sum of the following series
\begin{equation}\label{eq:G_tlawassum}
Z^{k}(t)\overset{\L}{=}\sqrt{\phi_0}\sum_{\ell=1}^\infty \dfrac{t^\ell 
	\sigma^\ell }{\ell !}G_\ell^k, \qquad\forall t\in \R_+.
\end{equation}
 where $(G_\ell^k)_{\ell\geq 1,k\geq 1}$ 
 are independent standard Gaussian random variables, independent of the initial condition $V_0^k$.

 This decomposition yields the covariance of $V^{k,(\infty)}$:
 \begin{equation}\label{eq:covVinfty}
\Cov\left[ V^{k,(\infty)}(t),V^{k,(\infty)}(s)  \right] = \phi_0 e^{-\lambda 
 (t+s)}I_0(2\sigma \sqrt{ts}) - e^{-\lambda(t+s)}\mathbb{E}(V_0)^2
 \end{equation}
It also shows that the only sources of randomness in the solution to \eqref{eq:EDS} (when $\gamma=0$) are the initial 
condition and  the family $(G_\ell^k)_{\ell\geq 1,k \geq 1}$ which does not depend on 
time.
The limiting process is thus an $\mathcal{F}_{0^+}$-measurable process. 
\end{remark}
\begin{remark}[Non independent increments]
It follows from the above that the increments of the process $\left(V^{k,(\infty)}(t)\right)_{t\in \R_+}$ are not independent since  for all $0\leq 
t_1<t_2\leq t_3<t_4$, 
\[\Cov\left[V^{k,(\infty)}(t_2)-V^{k,(\infty)}(t_1),V^{k,(\infty)}(t_4)-V^{k,(\infty)}(t_3)\right]\neq 0.\]
\end{remark}
\begin{remark}
Using Theorem~\ref{thm:cv} and \eqref{eq:G_tlawassum} we obtain the SDE satisfied by the process $(V^{k,(\infty)}(t))_{t\in \R_+}$: 
\begin{equation}\label{eq:Vtdiff}
\left\lbrace
\begin{aligned}
\rd V^{k,(\infty)}(t) &=-\lambda V^{k,(\infty)}(t) \rd t +H^{k}(t)\rd t\\
\L(V_0^{k,(\infty)})&=\loiinit.
\end{aligned}
\right.
\end{equation}
where $\left(H^{k}(t)\right)_{t\in \R_+}$ is the centered Gaussian process 
\[
\forall t\geq 0, \quad
H^k(t)=\sqrt{\phi_0} \sigma \sum_{\ell\geq 0}\frac{(\sigma t)^\ell}{\ell!}G_{\ell+1}^k.
\]
The standard Gaussian random variables $G_\ell^k$ have been introduced in  
\eqref{eq:G_tlawassum}. It is easily verified that
\[
\E\left[H^{k}(t)H^{k}(s)\right]=\phi_0\sigma^2 I_0\left(2\sigma\sqrt{ts}\right) (= \phi_0\sigma^2 ( 1 + \widetilde{I}_0(2\sigma\sqrt{ts})) ).
\]
\end{remark}
\begin{remark}[Long time behavior]
The function $\widetilde{I}_0$ behaves as an $\mathcal{O}\left(e^z/\sqrt{2\pi z}\right)$ as $z\to +\infty$. As a consequence of \eqref{eq:covVinfty}  $\sigma = \lambda$ is a critical value for the solution: if $\sigma>\lambda$ 
the solution \(V^1\) blows up when $t \to \infty$ while if $\sigma<\lambda$ it 
converges to its mean.
\end{remark}
\begin{remark}[Non Markov property]
The hypothesis of independence between $J$ and $V_0$ is crucial in the 
proof of Theorem~\ref{thm:cv}. Therefore the proof is not valid if we start the 
system at a time $t_1>0$, in other words, we cannot establish the existence 
of a process $\left(\bar{Z}^{k}(t)\right)_{t\in\R_+}$ with the same 
	law as $\left(Z^{k}(t)\right)_{t\in \R_+}$,  independent of 
$V^{k,(\infty)}(t_1)$, and such that
\begin{equation}\label{eq:partant_t1}
V^{k,(\infty)}(t_1+t)=e^{-\lambda t}\left(V^{k,(\infty)}(t_1)+\bar{Z}^{k}(t)\right).
\end{equation} 
\end{remark}
\subsection{Propagation of chaos}\label{Subsec:chaos}
An important consequence of Lemma~\ref{lem:1} and its Corollary~\ref{coro:pcomponents} is that the propagation of chaos property is satisfied by the mean field limit.
\begin{theorem}\label{thm:propagation1}
	For any finite number of labels $1 \leq k_1< k_2 < \cdots <k_p$,  and for all $t\in \R_+$, 
	the random processes
	 \(V^{k_1,(\infty)}(t)\), \(V^{k_2,(\infty)}(t)\), \(\cdots\), \(V^{k_p,(\infty)}(t)\) are \textbf{independent} and identically distributed.
\end{theorem}
\begin{proof}
  The proof follows directly from  Corollary~\ref{coro:pcomponents} and the proof of Theorem~\ref{thm:cv}.
\end{proof}

\section{Convergence of the particle system with an additive Brownian noise (\(\gamma\neq 0\))}\label{Sec:Noise}
In this section, we consider the general equation \eqref{eq:eds_matrix} with $\gamma \neq 0$.
\subsection{Mean field limit}
Here, the randomness is not entirely determined at time $t=0$, so we expect that the \mbox{$\mathcal{F}_{0^+}$-measurable} 
property of the limit which is satisfied by the mean field limit of 
Section~\ref{Sec:MainResult- without noise case} is no longer true. 
Considering $\gamma\neq 0$, the explicit solution of~\eqref{eq:eds_matrix} is 
\begin{equation}\label{eq:explicit_solution_MB}
V^{(N)}(t)=e^{-\lambda t} 
\exp\left(\dfrac{J^{(N)}t}{\sqrt{N}}\right)\left[V_0^{(N)}+\gamma\int_0^t e^{\lambda 
s}\exp\left(-\dfrac{J^{(N)}s}{\sqrt{N}}\right)\rd B(s)\right].
\end{equation}

\begin{theorem}\label{thm:cvnoise}
Under  hypotheses {\rm \ref{itm:H1} and  \ref{itm:H2}} on the matrix $J$ and the 
initial random vector $V_0^{(N)}$, for each $k\in \N^*$, $V^{k,(N)}(t)$  converges in 
law as $N\to +\infty$  to
\begin{equation}
V^{k,(\infty)}(t):= e^{-\lambda t}\left[V_0^{k}+\gamma \int_0^t e^{\lambda s}\rd B^{k}(s)+
Z^k(t)+\gamma A^k(t)
\right].
\end{equation}

The process $(Z^k(t))_{t\in \R_+}$ is the same centered Gaussian process as in Theorem~\ref{thm:cv}. It is independent of the initial condition $V_0^k$, of the Brownian motion $(B^k(t))_{t\in \R_+}$ and of the process $A^k(t)$. The process $(A^k(t))_{t\in \R_+}$ is a centered Gaussian process, also independent of $V_0^k$ and of $(B^k(t))_{t\in \R_+}$.
Its covariance writes
\[
\E\left[ A^1(t)A^1(s)  \right]=  \int_0^s e^{2\lambda u} \widetilde{I}_0(2\sigma\sqrt{(t-u)(s-u)}) \, d u\quad 0 \leq s \leq t
\]
\end{theorem}
\begin{proof}
  Without loss of generality we assume $k=1$.
  
The proof is in two parts. We first assume $\loiinit=\delta_0$ and show that, starting 
from $0$  (which implies that only the Brownian part of \eqref{eq:Vtsoluceexplicit} 
acts) the process converges in law to $\gamma\int_0^t e^{-\lambda (t-s)} \rd 
B^1(s)+\gamma e^{-\lambda t} A^1(t)$. 
In the second part we remove the condition $\nu_0 = \delta_0$.

\noindent
\textbf{Part 1: }Let $\loiinit=\delta_0$, then  the explicit solution of \eqref{eq:eds_matrix} is given by
\begin{equation}\label{eq:eds_gamma_nonzero}
V^{(N)}(t)=\gamma  \int_0^t e^{-\lambda (t- 
s)}\exp\left(\dfrac{J^{(N)}}{\sqrt{N}}(t-s)\right)\rd B(s).
\end{equation}
As before, we expand the exponential of $J$ and obtain
\begin{equation}\label{eq:weapplyLambdaon}
V^{(N)}(t)=\gamma e^{-\lambda t}\left[\int_0^te^{\lambda s}\rd B(s)+ \sum_{\ell\geq 1} 
\int_0^{t}e^{\lambda s}\left(\dfrac{\left(J^{(N)}\right)^\ell}{\sqrt{N}^\ell \, \ell 
!}(t-s)^\ell\right)\rd B(s)\right].
\end{equation}
For all $\ell \geq 1$, we introduce the $C^1$ function $\Lambda_\ell$ defined by
\begin{equation}\label{eq:Lambdal}
  \Lambda_\ell(t,s)=\frac{1}{\ell !}e^{\lambda s}(t-s)^\ell,
\end{equation}
and write
\[
  V^{(N)}(t)=\gamma e^{-\lambda t}\left[\int_0^te^{\lambda s}\rd B(s)+ \sum_{\ell\geq 
  1} \int_0^{t}\dfrac{\left(J^{(N)}\right)^\ell}
  {\sqrt{N}^\ell}\Lambda_\ell(t,s)\rd B(s)\right].
\]
We focus on the first coordinate of $V^{(N)}(t)$ and we introduce, for all $\ell\geq 
1$, the following notations.
\begin{equation}\label{eq:Utilde}
  U_\ell^{(N)}(t):=
   \transp{e_1}\left(J^{(N)}\right)^\ell \, \int_0^t \Lambda_\ell(t,s)d B(s).
\end{equation}
Then, we have
\[ V^{1,(N)}(t)=\gamma e^{-\lambda t} \left[\int_0^t e^{\lambda s}\rd B^1(s) + 
\sum_{\ell\geq 1} \frac{1}{\sqrt{N}^\ell} U^{(N)}_\ell(t)\right].\]
We next define
\[
A_n^{(N)}(t) :=  \sum_{\ell=1}^n \frac{1}{\sqrt{N}^\ell} U^{(N)}_\ell(t)
\]
For each $\ell \geq 1$ and $t \in \R_+$ the sequence $(Y^\ell_j(t))_{j \in \N^*}$, $Y^\ell_j(t)=\int_0^t \Lambda_\ell(t,s) dB^j(s)$, satisfies \eqref{eq:assumptionV0}. Indeed, by the law of large numbers
\begin{equation}\label{eq:phil}
  \lim_{N \to \infty} \frac{1}{N}\sum_{j=1}^N (Y^\ell_j(t))^2  = \E\left[\left( \int_0^t 
  \Lambda_\ell(t,s)\, dB^1(s)  \right)^2   \right]=
\int_0^t (\Lambda_\ell(t,s))^2 ds  := \phi_\ell(t)
\end{equation}
Moreover, for each $\ell \geq 1$, for each $t \geq 0$, the $Y^\ell_j(t)$ are independent centered Gaussian variables.

It follows from Remark \ref{rem:ind} and a slight modification of the proof of Lemma~\ref{lem:1} that the $n$-dimensional vector $\left(\frac{1}{\sqrt{N}}U_1^{(N)}(t)\cdots  \frac{1}{\sqrt{N}^n}U_n^{(N)}(t)  \right)$ converges in law to an $n$-dimensional centered Gaussian process with diagonal covariance matrix ${\rm diag}(\sigma^{2\ell} \phi_\ell(t)) $, $\ell=1,\cdots,n$. This process is independent of any finite subset of the Brownians $B^j$, hence of $\int_0^t e^{\lambda s}\rd B^1(s)$. It follows that $A_n^{(N)}(t)$ converges in law toward a centered Gaussian process $A_n(t)$ of variance $\sum_{\ell=1}^n \sigma^{2\ell} \phi_\ell(t)$.

Because $\sum_{\ell=1}^n \sigma^{2\ell}\phi_\ell(t)$ converges to $\sum_{\ell \geq 1} \sigma^{2\ell}\phi_\ell(t)$ and $A_n(t)$ is a centered Gaussian process, the Kolmogorov–Khinchin Theorem 
 commands that $A_n(t)$ converges in law when $n \to \infty$ to a centered 
 Gaussian process $A^1(t)$ independent of the Brownian $B^1$, with covariance 
 $\Lambda^2(t):=\sum_{\ell \geq 1} \sigma^{2\ell}\phi_\ell(t)$.
By \eqref{eq:Lambdal} and \eqref{eq:I0tilde} we have
\[
\Lambda^2(t) = \int_0^t e^{2\lambda u} \widetilde{I}_0(2\sigma(t-u)) \, d u,
\]
It remains to prove that $A^{(N)}(t)$ converges in law to $A^1(t)$.
 According to 
 \cite[Th. 25.5]{billingsley1995}, to obtain the weak convergence of 
 $A^{(N)}(t)$ to $A(t)$ as $N\to +\infty$,  it is sufficient to show that for 
 all $t\in \R_+$
 \begin{equation}\label{eq:limnlimsupN_bis}
 \lim_{n\to +\infty}\limsup_{N\to +\infty}\P\left[\mid A_n^{(N)}(t)-A^{(N)}(t)\mid\geq \varepsilon\right]=0,\qquad \forall \eps>0.
 \end{equation}
 By Markov inequality, we have 
 \[
 \P\left[\mid A_n^{(N)}(t)-A^{(N)}(t)\mid\geq \eps \right]\leq \frac{1}{\varepsilon}\E\left[\left| A_n^{(N)}(t)-A^{(N)}(t)\right|\right],
 \]
 and
 \[
 \E\left[\mid A_n^{(N)}(t)-A^{(N)}(t)\mid\right]\leq  \sum_{\ell\geq n+1}\E\left[\left|\frac{1}{\sqrt{N}^l} {U}_\ell^{(N)}(t)\right|\right].
\]
 We know from the beginning of the proof that for all  $t\in \R_+$ and for all $\ell\in 
 \N^*$, , 
 \[
   \lim_{N\to +\infty}\frac{1}{\sqrt{N}^l} U_\ell^{(N)}(t):= U_\ell(t) \overset{\mathcal{L}}{=} \mathcal{N}(0,\sigma^{2\ell} \phi_\ell(t))
 \]
  Then  the law of $|{U}_\ell(t)|$ for all $t \in \R_+$  is a half normal distribution, and 
 \begin{align*}
\E\left[|{U}_\ell(t)|\right]=\sqrt{\dfrac{2}{\pi}}\sigma^{\ell}\left( \phi_\ell(t) \right)^{1/2}
 \end{align*}
 Then, 
 \[
   \limsup_{N\to+\infty}\E\left[\left| A_n^{(N)}(t)-A^{(N)}(t)\right|\right] \leq  
   \frac{1}{\varepsilon}\sqrt{\dfrac{2}{\pi}} \sum_{\ell > n} \sigma^{\ell}\phi_{\ell}(t)
 \]
 The right hand side of this inequality goes to zero when $n \to \infty$ and \eqref{eq:limnlimsupN_bis} follows. This concludes the first part of the proof.

 In order to prove that the process $(A^1(t))_{t\in \R_+}$ is Gaussian we proceed 
 exactly as in the proof of Theorem~\ref{thm:cv} and show that $aA^1(t)+bA^1(s)$ is 
 Gaussian for all reals $a$ and $b$ and all $s < t \in \R_+$.  Indeed, the 
 previous proof commands that $aA^{(N)}(t)+bA^{(N)}(s)$ converges in law when 
 $N \to \infty$ toward a centered Gaussian random variable of variance 
 $a^2\Lambda^2(t)+2ab \Lambda(t,s)+b^2\Lambda^2(s)$, where
   \[
\Lambda(t,s) =  \int_0^s e^{2\lambda u} \widetilde{I}_0(2\sigma\sqrt{(t-u)(s-u)}) \, d u\quad 0 \leq s \leq t  
\]

 \textbf{Part 2: }

 We remove the assumption $\nu_0 = \delta_0$. A slight modification of the proof of Lemma~\ref{lem:1} shows that
 the $2n$-dimensional vector
 \[
\left(\frac{1}{\sqrt{N}}U_1^{(N)},\cdots,\frac{1}{\sqrt{N}^n}U_n^{(N)},\frac{1}{\sqrt{N}}U_1^{(N)}(t),\cdots,\frac{1}{\sqrt{N}^n}U_n^{(N)}(t)\right),
\]
where $U_\ell^{(N)}$ is defined by \eqref{eq:ZandU} and  $U_\ell^{(N)}(t)$ by \eqref{eq:Utilde}, converges in law when $N \to \infty$ to the $2n$-dimensional centered Gaussian vector with covariance ${\rm diag}(\sigma^{2}\phi_0,\cdots,\sigma^{2n}\phi_0,\sigma^{2}\phi_1,\cdots,\sigma^{2n}\phi_n)$,
where $\phi_\ell$, $\ell \geq 1$ is defined by \eqref{eq:phil}.

We conclude that $Z_n^{(N)}(t)+A_n^{(N)}(t)$ converges in law when $N \to \infty$ 
to $Z_n(t)+A_n(t)$, and that $Z_n(t)$ and $A_n(t)$ are independent. The 
convergence of $Z_n(t)+A_n(t)$ to $Z^1(t)+A^1(t)$ follows again from the 
Kolmogorov–Khinchin Theorem.
\end{proof}
\begin{remark}
Note that the process $A^1(t)$ does not have independent increments.
\end{remark}
\subsection{Propagation of chaos}
As in the case without noise, propagation of chaos occurs.
\begin{theorem}\label{thm:prop_chaos_noise}
For any finite number of labels $k_1< k_2 < \cdots <k_p$,  and for all $t\in \R_+$, the random processes 
\(V^{k_1,(\infty)}(t)\), \(V^{k_2,(\infty)}(t)\), \(\cdots\), \(V^{k_p,(\infty)}(t)\) 
are \textbf{independent} and identically distributed.
\end{theorem}
\begin{proof}
   The proof follows directly from the following extension of Corollary~\ref{coro:pcomponents} and the proof of Theorem~\ref{thm:cv}.
  \begin{corollary}\label{coro:pcomponentswithB}
Under the same assumptions as in Lemma~\ref{lem:1} (see also Remark~\ref{rem:ind}),
 for all integers $p > 1$ and  $1 \leq k_1 < k_2 < \cdots < k_p$ the $2mp$-dimensional vector obtained by concatenating the two $mp$-dimensional vectors
 \begin{multline*}
\big(\frac{1}{\sqrt{N}^{\ell_1}}U_{\ell_1}^{k_1,(N)}, \cdots , 
\frac{1}{\sqrt{N}^{\ell_m}}U_{\ell_m}^{k_1,(N)},\\ \frac{1}{\sqrt{N}^{\ell_1}}U_{\ell_1}^{k_2,(N)}, \cdots , 
\frac{1}{\sqrt{N}^{\ell_m}}U_{\ell_m}^{k_2,(N)}, \cdots, \frac{1}{\sqrt{N}^{\ell_1}}U_{\ell_1}^{k_p,(N)}, \cdots , 
\frac{1}{\sqrt{N}^{\ell_m}}U_{\ell_m}^{k_p,(N)}\bigg),
\end{multline*}
where $U_{\ell_j}^{k_i,(N)}=\transp{e_{k_i}}J^{\ell_j}V_0$, $i=1,\cdots,p$, $j=1,\cdots,m$  and
 \begin{multline*}
\bigg(\frac{1}{\sqrt{N}^{\ell_1}}U_{\ell_1}^{k_1,(N)}(t), \cdots , 
\frac{1}{\sqrt{N}^{\ell_m}}U_{\ell_m}^{k_1,(N)}(t),\\ \frac{1}{\sqrt{N}^{\ell_1}}U_{\ell_1}^{k_2,(N)}(t), \cdots , 
\frac{1}{\sqrt{N}^{\ell_m}}U_{\ell_m}^{k_2,(N)}(t), \cdots, \frac{1}{\sqrt{N}^{\ell_1}}U_{\ell_1}^{k_p,(N)}(t), \cdots , 
\frac{1}{\sqrt{N}^{\ell_m}}U_{\ell_m}^{k_p,(N)}(t)\bigg),
\end{multline*}
where $ U_{\ell_j}^{k_i,(N)}(t):=
   \transp{e_{k_i}}J^{\ell_j} \, \int_0^t \Lambda_{\ell_j}(t,s)d B(s)$,  $i=1,\cdots,p$, $j=1,\cdots,m$, 
converges in law as $N\to +\infty$, to an $2mp$-dimensional Gaussian random vector of diagonal 
covariance matrix $\mathrm{diag}(\sigma^{2\ell_i}\phi_0)$, $i=1,\cdots,m$ repeated $p$ times, and $\mathrm{diag}(\sigma^{2\ell_i}\phi_{\ell_i}(t))$, $i=1,\cdots,m$ repeated $p$ times.
\end{corollary}
\begin{proof}
Follows from the one of Lemma~\ref{lem:1}.
\end{proof}
\end{proof}

\section*{Acknowledgements}
This project/research has received funding from the European Union’s Horizon 2020 
Framework Programme for Research and Innovation under the Specific Grant 
Agreement No. 945539 (Human Brain Project SGA3) and the Specific Grant 
Agreement No. 785907 (Human Brain Project SGA2).

\end{document}